\documentclass[
reqno]{amsart}
\usepackage{latexsym,amsmath,amssymb,amscd}
\usepackage[all]{xy}

\title[Differentiable vectors for contragredient representations]
{On the differentiable vectors for contragredient representations}
\author{Ingrid Belti\c t\u a}
\author{Daniel Belti\c t\u a}
\address{Institute of Mathematics ``Simion Stoilow'' of the Romanian Academy,
P.O. Box 1-764, Bucharest, Romania}
\email{Ingrid.Beltita@imar.ro, ingrid.beltita@gmail.com}
\email{Daniel.Beltita@imar.ro, beltita@gmail.com}
\date{June 22, 2013}

\begin{document}

\begin{abstract}
We establish a few simple results on contragredient representations of Lie groups, 
with a view toward applications to the abstract characterization of 
some spaces of pseudo-differential operators. 
In particular, this method provides an abstract approach to 
J.~Nourrigat's recent description of the norm closure of the 
pseudo-differential operators of order zero. 
\end{abstract}


\makeatletter
\@addtoreset{figure}{section}
\def\thefigure{\thesection.\@arabic\c@figure}
\def\fps@figure{h,t}
\@addtoreset{table}{bsection}

\def\thetable{\thesection.\@arabic\c@table}
\def\fps@table{h, t}
\@addtoreset{equation}{section}
\def\theequation{
\arabic{equation}}
\makeatother

\newcommand{\bfi}{\bfseries\itshape}

\newtheorem{theorem}{Theorem}
\newtheorem{corollary}[theorem]{Corollary}
\newtheorem{definition}[theorem]{Definition}
\newtheorem{example}[theorem]{Example}
\newtheorem{lemma}[theorem]{Lemma}
\newtheorem{notation}[theorem]{Notation}
\newtheorem{problem}[theorem]{Problem}
\newtheorem{proposition}[theorem]{Proposition}
\newtheorem{remark}[theorem]{Remark}

\numberwithin{theorem}{section}
\numberwithin{equation}{section}

\renewcommand{\1}{{\bf 1}}

\newcommand{\Ad}{{\rm Ad}}
\newcommand{\ad}{{\rm ad}}
\newcommand{\Ci}{{\mathcal C}^\infty}

\newcommand{\Aut}{{\rm Aut}}
\newcommand{\ph}{\text{\bf P}}
\newcommand{\de}{{\rm d}}
\newcommand{\ee}{{\rm e}}
\newcommand{\ie}{{\rm i}}
\renewcommand{\Im}{{\rm Im}}
\newcommand{\Ker}{{\rm Ker}\,}
\newcommand{\Lie}{\textbf{L}}
\newcommand{\Op}{{\rm Op}}
\newcommand{\Ran}{{\rm Ran}\,}

\newcommand{\Tr}{{\rm Tr}\,}

\newcommand{\CC}{{\mathbb C}}
\newcommand{\HH}{{\mathbb H}}
\newcommand{\NN}{{\mathbb N}}
\newcommand{\RR}{{\mathbb R}}

\newcommand{\Bc}{{\mathcal B}}
\newcommand{\Cc}{{\mathcal C}}
\newcommand{\Dc}{{\mathcal D}}
\newcommand{\Hc}{{\mathcal H}}
\newcommand{\Kc}{{\mathcal K}}
\newcommand{\Xc}{{\mathcal X}}
\newcommand{\Yc}{{\mathcal Y}}
\newcommand{\Ag}{{\mathfrak A}}

\renewcommand{\gg}{{\mathfrak g}}
\newcommand{\Sg}{{\mathfrak S}}
\newcommand{\zg}{{\mathfrak z}}


\maketitle


\section{Introduction}

In this note we study the abstract characterization of 
some spaces of pseudo-differential operators by using 
a few simple results on the contragredients of Banach space representations of Lie groups. 
The applicability of the method based on a contragredient representation is 
due to the fact that such a representation may be discontinuous even if the original representation is continuous; 
see for instance the representation \eqref{appl_eq1} below, which is discontinuous if $r=\infty$. 
In particular, we provide an abstract approach to 
J.~Nourrigat's recent description \cite{No12} of the norm closure of the 
pseudo-differential operators of order zero 
(see Example~\ref{E1} below)
and we also bring additional information on some results from the earlier literature. 

\subsection*{Preliminaries}
For any complex Banach space $\Yc$ we denote by $\Yc^*$ its topological dual  
and by $\Bc(\Yc)^\times$ the group of invertible elements 
in the Banach algebra $\Bc(\Yc)$ of all bounded linear operators.  
If $G$ is any group, then a Banach space representation of $G$ 
is a 
group homomorphism $\pi\colon G\to\Bc(\Yc_\pi)^\times$, where $\Yc_\pi$ is a complex Banach space. 
The contragredient representation of $\pi$ is the representation 
$$\pi^*\colon G\to\Bc(\Yc_\pi^*)^\times, \quad \pi^*(g):=\pi(g^{-1})^*,$$ 
so that $\Yc_{\pi^*}:=\Yc_\pi^*$. 
If $\sup\limits_{g\in G}\Vert\pi(g)\Vert<\infty$, then we say that $\pi$ is uniformly bounded, 
and in this case also $\pi^*$ is uniformly bounded. 

Now assume that $G$ is a topological group and for 
the uniformly bounded representation $\pi\colon G\to\Bc(\Yc_\pi)^\times$ 
define $\Yc_{\pi_0}:=\{x\in\Yc_\pi\mid\pi(\cdot)x\in\Cc(G,\Yc_\pi)\}$, 
where $\Cc$ indicates the space of continuous mappings. 
Then $\Yc_{\pi_0}$ is a closed linear subspace of $\Yc$ since $\pi$ is uniformly bounded, 
and moreover $\Yc_{\pi_0}$ is invariant under $\pi$. 
Hence we obtain a strongly continuous representation 
$\pi_0\colon G\to\Bc(\Yc_{\pi_0})$, $\pi_0(g):=\pi(g)\vert_{\Yc_{\pi_0}}$. 
By using this construction for the contragredient representation, 
we define 
$$\Yc_{\pi^*_0}:=\{\xi\in\Yc_\pi^*\mid\pi^*(\cdot)\xi\in\Cc(G,\Yc_\pi^*)\}
=\{\xi\in\Yc_\pi^*\mid\lim_{g\to\1}\Vert\pi^*(g)\xi-\xi\Vert=0\}$$
and 
$$\pi^*_0\colon G\to\Bc(\Yc_{\pi^*_0})^\times, \quad\pi^*_0(g):=\pi^*(g)\vert_{\Yc_{\pi^*_0}}.$$
If moreover $G$ is a Lie group, then we also define 
$\Yc_{\pi}^k:=\{y\in\Yc\mid\pi(\cdot)y\in\Cc^k(G,\Yc)\}$ for every integer $k\ge0$, 
so in particular $\Yc_{\pi}^0=\Yc_{\pi_0}$. 
Moreover, if the representation~$\pi$ is strongly continuous, that is, $\Yc=\Yc_{\pi_0}$, 
then for every $k\ge 1$ and every basis $\{X_1,\dots,X_m\}$ in the Lie algebra $\gg$ of~$G$ 
we have 
\begin{equation}\label{pre_eq1}
\Yc_{\pi}^k=\bigcap_{1\le j_1,\dots,j_k\le m}\Dc(\de\pi(X_{j_1})\cdots\de\pi(X_{j_k}))
\end{equation}
(see for instance \cite[Th. 9.4]{Ne10}).  
Here and in what follows we denote by $\Dc(T)$ the domain of any unbounded operator~$T$. 

\section{The main abstract results}

The following theorem can be regarded as a version of \eqref{pre_eq1} 
for some discontinuous representations of Lie groups, 
namely for the contragredient of any uniformly bounded and strongly continuous representation.  

\begin{theorem}\label{main}
Let $G$ be a Lie group with a strongly continuous representation $\pi\colon G\to\Bc(\Yc)$ 
which is also assumed to be uniformly bounded. 
If $\{X_1,\dots,X_m\}$ is any basis in the Lie algebra $\gg$ of~$G$,   
then  
$$\Yc_{\pi^*}^k
\subseteq
\bigcap_{1\le j_1,\dots,j_k\le m}\Dc(\de\pi(X_{j_1})^*\cdots\de\pi(X_{j_k})^*)
\subseteq
\Yc_{\pi^*}^{k-1}
$$
for every integer $k\ge 1$, 
and the above inclusions could simultaneously be strict. 
\end{theorem}

For proving the theorem it will be convenient to use the notation 
$$\Cc^k(\pi^*):=\bigcap_{1\le j_1,\dots,j_k\le m}\Dc(\de\pi(X_{j_1})^*\cdots\de\pi(X_{j_k})^*)$$
for an arbitrary integer $k\ge 1$. 
It is clear that $\Cc^1(\pi^*)\supseteq\Cc^2(\pi^*)\supseteq\cdots$. 

The proof will be based on the following auxiliary result, 
which should be thought of as an embedding lemma on abstract Sobolev spaces. 

\begin{lemma}\label{L1}
We have $\Cc^1(\pi^*)\subseteq\Yc_{\pi^*_0}$. 
\end{lemma}

\begin{proof}
For every $X\in\gg$ let us denote $\gamma_X\colon\RR\to G$, $\gamma_X(t):=\exp_G(tX)$. 
It follows by \cite[Th. 1.3.1]{vNe92} that 
\begin{equation}\label{L1_proof_eq1}
\Dc(\de\pi(X)^*)\subseteq\Yc_{\pi^*\circ\gamma_X}=\{\xi\in\Yc^*\mid\pi^*(\gamma_X(\cdot))\xi\in\Cc(\RR,\Yc^*)\}
\end{equation}
for arbitrary $X\in\gg$. 
On the other hand, we have 
\begin{equation}\label{L1_proof_eq2}
\Yc_{\pi^*}=\bigcap_{j=1}^m\Yc_{\pi^*\circ\gamma_{X_j}}
\end{equation}
since the inclusion $\subseteq$ is obvious while the inclusion $\supseteq$ 
holds true for the following reason. 
For all $t_1,\dots,t_m\in\RR$ and $\xi\in\Yc^*$ we have 
$$\begin{aligned}
\Vert\pi^*(\gamma_{X_1}(t_1) & \cdots\gamma_{X_m}(t_m))\xi-\xi\Vert \\
& \le \sum_{j=1}^m \Vert\pi^*(\gamma_{X_1}(t_1)\cdots\gamma_{X_{j-1}}(t_{j-1}))(\pi^*(\gamma_{X_j}(t_j))\xi-\xi)\Vert \\
& \le M \sum_{j=1}^m \Vert\pi^*(\gamma_{X_j}(t_j))\xi-\xi\Vert 
\end{aligned}$$
where $M:=\sup\limits_{g\in G}\Vert\pi(g)\Vert$. 
Since $\{X_1,\dots,X_m\}$ is a basis in the Lie algebra $\gg$, 
it follows that the mapping $\RR^m\to G$, $(t_1,\dots,t_m)\mapsto \gamma_{X_1}(t_1)\cdots\gamma_{X_m}(t_m)$, 
is a local diffeomorphism at $0\in\RR^m$, and then the above estimate shows that 
for every $\xi\in \bigcap\limits_{j=1}^m\Yc_{\pi^*\circ\gamma_{X_j}}$ 
we have $\lim\limits_{g\to\1}\Vert\pi^*(g)\xi-\xi\Vert=0$, hence $\xi\in\Yc_{\pi^*}$. 
This completes the proof of \eqref{L1_proof_eq2}.

Now, 
since $\Dc(\de\pi(X_1)^*)\cap\cdots\cap\Dc(\de\pi(X_m)^*)=\Cc^1(\pi^*)$, 
the assertion follows by \eqref{L1_proof_eq1} and \eqref{L1_proof_eq2}. 
\end{proof}

\begin{proof}[Proof of Theorem~\ref{main}]
By using Lemma~\ref{L1} and \cite[Lemma 1.1]{Po72} we obtain 
$$\Cc^k(\pi^*)\subseteq\bigcap_{1\le j_1,\dots,j_{k-1}\le m}\Dc(\de\pi_0^*(X_{j_1})\cdots\de\pi_0^*(X_{j_{k-1}}))
=\Yc_{\pi^*}^{k-1}$$
where the latter equality follows by using \eqref{pre_eq1}  
for the strongly continuous representation~$\pi^*_0$. 
The inclusion  $\Yc_{\pi^*}^k\subseteq \Cc^k(\pi^*)$ can be easily proved by using \eqref{pre_eq1}
and the fact that for every $X\in \gg$ we have 
$\Dc(\de \pi_0^*(X)) \subset \Dc(\de\pi(X)^*)$ and 
$ \de\pi(X)^*\vert_{\Dc(\de \pi_0^*(X))}= \de \pi_0^*(X)$.

We now prove by example that the inclusion in the statement can be strict for $k=1$. 
To this end let $G=\RR$, $\Yc$ be the space of trace-class operators on $L^2(\RR)$, 
and consider the regular representation $\rho\colon\RR\to \Bc(L^2(\RR))$, $\rho(t)f=f(\cdot+t)$. 
Then define $\pi\colon\RR\to\Bc(\Yc)$, $\pi(t)A=\rho(t)A\rho(t)^{-1}$ 
and for every $\phi\in L^\infty(\RR)$ denote by $\phi(Q)$ the multiplication-by-$\phi$ operator on $L^2(\RR)$, 
so that $\phi(Q)\in\Bc(L^2(\RR))\simeq\Yc^*$. 
It was noted in \cite[Ex. 6.2.7]{ABG96} that  $\phi(Q)\in\Yc_{\pi^*}^k$ 
if and only the first $k$ derivatives of $\phi$ exist, are bounded, 
and the $k$-th derivative is also uniformly continuous on~$\RR$.  

On the other hand, if we denote by $P=-\ie \frac{\de}{\de t}$ the infinitesimal generator of $\rho$, 
then it is easily checked that $\phi(Q)\in\Cc^1(\pi^*)$ if and only if the commutator $[\phi(Q),P]$ belongs to $\Bc(L^2(\RR))$, 
hence by using also \cite[Prop. 5.1.2(b)]{ABG96} and again \cite[Ex. 6.2.7]{ABG96} 
we see that the latter commutator condition is equivalent to the fact that $\phi$ is bounded and satisfies the Lipschitz condition 
globally on~$\RR$. 
Therefore there exist $\phi, \psi \in L^\infty(\RR)$ such that $\phi(Q)\in\Cc^1(\pi^*)\setminus\Yc_{\pi^*}^1$ and
$\psi(Q) \in  \Yc_{\pi^*}^2\setminus \Cc^1(\pi^*)$.
This completes the proof. 
\end{proof}

\begin{corollary}\label{C1}
In the setting of Theorem~\ref{main}, the linear subspace 
$$\bigcap_{k\ge1}\bigcap_{1\le j_1,\dots,j_k\le m}\Dc(\de\pi(X_{j_1})^*\cdots\de\pi(X_{j_k})^*)$$
is dense in $\Yc_{\pi^*_0}$. 
\end{corollary}

\begin{proof}
It follows by Theorem~\ref{main} that this linear subspace 
is equal to the space of smooth vectors for the strongly continuous representation 
$\pi^*_0$, hence it is dense in the representation space $\Yc_{\pi^*_0}$ 
(see \cite{Ga47}). 
\end{proof}

\section{Applications}

In this section we will develop a more general version of the example used in the proof of Theorem~\ref{main}. 
Let $G$ be a Lie group with a continuous unitary representation 
$\rho\colon G\to\Bc(\Hc)$. 
If $1\le p<\infty$, denote by $\Sg_p(\Hc)$ the $p$-th Schatten ideal, 
and let $\Sg_\infty(\Hc):=\Bc(\Hc)$ and $\Sg_0(\Hc)$ be the ideal of all compact operators on $\Hc$. 
It is well known that if $p,q\in\{0\}\cup[1,\infty]$ with $\frac{1}{p}+\frac{1}{q}=1$ and 
$p\ne\infty$, then there exists an isometric linear isomorphism 
$\Sg_p(\Hc)^*\simeq\Sg_q(\Hc)$ defined by the duality pairing 
$$\langle\cdot,\cdot\rangle\colon \Sg_q(\Hc)\times\Sg_p(\Hc)\to\CC,\quad \langle Y,V\rangle:=\Tr(YV). $$
The representation $\rho^{(q)}$ can thus be regarded as the contragredient representation of 
the \emph{strongly continuous} representation $\rho^{(p)}$, 
where  
\begin{equation}\label{appl_eq1}
(\forall r\in\{0\}\cup[1,\infty])\quad \rho^{(r)}\colon G\to \Bc(\Sg_r(\Hc)), \quad \rho^{(r)}(g)Y=\rho(g)Y\rho(g)^{-1}
\end{equation}
(see also \cite{BB10}). 

Here is a consequence of the results from the previous section. 
In the special case of the Heisenberg group, 
this establishes a direct relationship between the classical characterizations of pseudo-differential operators 
from \cite{Be77} and \cite{Co79}. 

\begin{corollary}\label{C2}
In the above setting, pick any basis $\{X_1,\dots,X_m\}$ in the Lie algebra~$\gg$ of $G$.  
Assume $1\le q\le\infty$ and denote 
$$\Psi_q(\rho):=\{Y\in\Sg_p(\Hc)\mid \rho^{(q)}(\cdot)Y\in\Ci(G,\Sg_p(\Hc))\} .$$ 
Then the following assertions hold: 
\begin{enumerate}
\item\label{C2_item1} 
The linear subspace $\Psi_q(\rho)$ is precisely the set of all $Y\in\Sg_q(\Hc)$ 
for which 
$$[\de\rho(X_{j_1}),\dots,[\de\rho(X_{j_k}),Y]\dots]\in \Sg_q(\Hc)$$ 
for arbitrary $k\ge 1$ and $j_1,\dots,j_k\in\{1,\dots,m\}$. 
\item\label{C2_item2} 
If $1\le q<\infty$, then $\Psi_q(\rho)$ is dense in $\Sg_q(\Hc)$. 
If $q=\infty$, then $\Psi_\infty(\rho)$ contains the ideal of compact operators on $\Hc$ 
and is dense in the norm-closed subspace $\{Y\in\Bc(\Hc)\mid \rho^{(\infty)}(\cdot)Y\in\Cc(G,\Bc(\Hc))\}$ 
of $\Bc(\Hc)$. 
\end{enumerate}
\end{corollary}

\begin{proof}
We have that 
$$ \Cc^1(\rho^{(q)})= \{ Y \in \Sg_q(\Hc) \mid [\de \rho(X_j), Y] \in \Sg_q(\Hc)\, \text{for}\, j=1, \dots, m\}.
$$
Then both assertions are special cases of Theorem~\ref{main} and Corollary~\ref{C1}. 
\end{proof}

We can now prove a corollary which shows that the first two conditions in \cite[Th. 1]{Me00} are equivalent 
irrespective of the unitary representation involved therein. 
This also shows that the $\Ci$ part of the 
relation between differentiability and existence of commutators
suggested after \cite[Eq. (8.4)]{Co95}  
holds true although the $\Cc^1$ part of that suggestion fails to be true, 
since the following corollary would be false with the class $\Ci$ 
replaced by $\Cc^k$ for any $k<\infty$. 
In fact, recall from the proof of Theorem~\ref{main} that the corresponding inclusions are strict 
in a special instance of the present setting, which is precisely the special instance referred to in \cite{Co95}. 

\begin{corollary}\label{C3}
If $Y\in\Bc(\Hc)$ then the above mapping $\rho^{(\infty)}(\cdot)Y\colon G\to\Bc(\Hc)$ 
is of class~$\Ci$ with respect to the norm operator topology on $\Bc(\Hc)$ 
if and only if it is~$\Ci$ with respect to the strong operator topology. 
\end{corollary}

\begin{proof}
The mapping $\rho^{(\infty)}(\cdot) Y\colon G\to\Bc(\Hc)$ is smooth with respect to any topology on $\Bc(\Hc)$ 
if and only if it is smooth on any neighborhood of $\1\in G$. 
On the other hand, just as in the proof of \cite[Prop. 5.1.2(b)]{ABG96}, 
one can see that this mapping is smooth with respect to the strong operator topology on $\Bc(\Hc)$ 
if and only if the iterated commutator condition in Corollary~\ref{C2}\eqref{C2_item1} is satisfied, 
hence the conclusion follows by Corollary~\ref{C2}\eqref{C2_item1}, where the smoothness 
of $\rho^{(\infty)}(\cdot)Y$ is understood with respect to the norm operator topology on $\Sg_\infty(\Hc)=\Bc(\Hc)$.
\end{proof}

\begin{example}\label{E1}
\normalfont
Let $G=\HH_{2n+1}$ be the $(2n+1)$-dimensional Heisenberg group 
with the Schr\"odinger representation $\rho\colon G\to\Bc(\Hc)$. 
As recalled in \cite{No12} for $1\le p\le\infty$, 
the set $\Psi_p(\rho)$ of the above Corollary~\ref{C2} 
is precisely the set of pseudo-differential operators on $L^2(\RR^n)$ 
corresponding to the space of symbols 
$$\{a\in\Ci(\RR^{2n})\mid(\forall\alpha\in\NN^{2n})\ \partial^\alpha a\in L^p(\RR^{2n})\}$$ 
(see also \cite{BB12} for similar results on more general nilpotent Lie groups). 
Thus our Corollary~\ref{C2} leads to the main results of \cite{No12}. 
\end{example} 

\begin{example}
\normalfont
The above Corollary~\ref{C2} also provides additional information 
in  the setting of pseudo-differential operators on a compact manifold acted on by a Lie group, 
as studied for instance in \cite{Ta97} and \cite{Me00}. 
Thus, it follows that the notions of $U$-smoothness and $\Ag$-smoothness 
from \cite[Sect. 2]{Ta97} actually coincide. 
\end{example}

\begin{remark}
\normalfont
It would be interesting to extend the above result of Ex.~\ref{E1}
to the setting of 
the magnetic Weyl calculus of \cite{IMP10}.  
Such an extension is likely to require infinite-dimensional Lie groups. 
\end{remark}

\textbf{Acknowledgment.}  
This research has been partially supported by the Grant
of the Romanian National Authority for Scientific Research, CNCS-UEFISCDI,
project number PN-II-ID-PCE-2011-3-0131.



\end{document}